\documentclass[11pt]{article}

\usepackage[a4paper,total={150mm,237mm},left=30mm,top=30mm]{geometry}
\usepackage{amsmath,amssymb,amsthm,mathtools}
\usepackage{graphicx}
\usepackage[dvipsnames]{xcolor}
\usepackage{hyperref}

\newcommand{\R}{\mathbb{R}}
\newcommand{\CC}{\mathcal{C}}
\newcommand{\PP}{\mathcal{P}}
\newcommand{\sphere}[1]{\mathbb{S}^{#1}}
\newcommand{\SOn}[1]{\operatorname{SO}(#1)}
\newcommand{\On}[1]{\operatorname{O}(#1)}
\newcommand{\SO}{\operatorname{SO}}
\newcommand{\Prob}{\mathbb{P}}
\newcommand{\vct}[1]{#1}
\newcommand{\zero}{\vct{0}}
\newcommand{\link}{\operatorname{link}}
\newcommand{\pos}{\operatorname{pos}}
\newcommand{\lin}{\operatorname{lin}}
\newcommand{\relint}{\operatorname{relint}}
\newcommand{\sgn}{\operatorname{sgn}}

\newcommand{\eps}{\varepsilon}
\newcommand{\inter}{\operatorname{int}}
\newcommand{\cl}{\operatorname{cl}}
\newcommand{\bd}{\operatorname{bd}}

\theoremstyle{plain}
\newtheorem{theorem}{Theorem}[section]
\newtheorem{proposition}[theorem]{Proposition}
\newtheorem{lemma}[theorem]{Lemma}
\newtheorem{corollary}[theorem]{Corollary}
\theoremstyle{definition}
\newtheorem{definition}[theorem]{Definition}
\newtheorem{remark}[theorem]{Remark}

\title{Monotone invariant valuations on convex cones}
\author{Martin Lotz\\[0.5ex]
  {\small Mathematics Institute, University of Warwick}\\
  {\small Coventry CV4 7AL, UK}}
\date{September 2026}

\begin{document}
\maketitle

\begin{abstract}
We classify the monotone $\SOn d$-invariant valuations, $d\ge2$, on
the space of all closed convex cones.  They are precisely the linear combinations
of the conic intrinsic volumes with nondecreasing coefficients, or equivalently, the
nonnegative linear combinations of
intrinsic volume tails, up to an additive constant.  On nonzero pointed
cones we obtain the corresponding characterisation by Grassmann angles,
settling a conjecture of McMullen.
Every such valuation is automatically continuous in the spherical Hausdorff
topology and is $\On d$-invariant.  Our proof builds on
the signed simplex and averaging arguments of Wang and Wu, but does
not rely on the continuous classification theorem.
\end{abstract}

\section{Introduction}\label{sec:introduction}
Hadwiger's classification theorem is a cornerstone of convex and integral geometry, providing a complete classification of continuous, translation-invariant, and rotation-invariant valuations on convex bodies in Euclidean space $\R^d$ in terms of the intrinsic volumes. These are a family of geometric invariants that generalise volume, surface area, and Euler characteristic.
Conic (or equivalently, spherical) analogues of the intrinsic volumes have a rich history and have found a myriad of applications in areas such as compressive sensing, optimisation, and statistics.
Recently, Knoerr~\cite[Theorem~A]{Knoerr:Spherical-Hadwiger} proved
the spherical analogue of Hadwiger's theorem, settling a decades-old conjecture. Independently, an alternative proof was later provided by Wang and Wu~\cite{WangWu:Spherical-Hadwiger}.
In its conic formulation, the theorem states that every continuous,
$\SOn d$-invariant valuation on the space of closed convex cones in $\R^d$
is a linear combination of the conic intrinsic volumes $v_0,\ldots,v_d$.

In this paper, we are concerned with \emph{monotone} $\SOn d$-invariant valuations on the space $\CC(\R^d)$ of closed convex cones in $\R^d$.
The monotone part of Problem~49 of Gruber and Schneider~\cite{Gruber1979},
attributed to McMullen, asks for a characterisation of monotone,
rotation-invariant valuations on pointed cones.  We show that on all
nonzero pointed closed cones they are precisely, up to an additive
constant, the nonnegative linear combinations of Grassmann angles.

\begin{theorem}\label{thm:global-classification}
  Let $d\ge2$.  A valuation $\mu$ on $\CC(\R^d)$ is monotone and
  $\SOn d$-invariant if and only if
  \begin{equation}\label{eq:global-classification}
    \mu=\sum_{k=0}^{d}a_kv_k,
    \qquad a_0\le a_1\le\cdots\le a_d.
  \end{equation}
  Every such valuation is continuous and
  $\On d$-invariant.
  \end{theorem}

For Grassmann angles $U_j$, related to the intrinsic volumes by Crofton's
formula (Proposition~\ref{prop:crofton}), we obtain the following
classification result.

  \begin{corollary}\label{cor:problem49}
    On the nonzero pointed cones, the
    monotone $\SOn d$-invariant valuations
    are precisely
    \[
      a_0U_0+a_1U_1+\cdots+a_{d-1}U_{d-1},
      \qquad a_0\in\R,\quad a_1,\ldots,a_{d-1}\ge0.
    \]
    \end{corollary}

\begin{remark}\label{rem:monotonicity-domains}
Note that in Corollary~\ref{cor:problem49}, both the valuation property and
  monotonicity are assumed only on the pointed domain, with no assumption
  of a monotone extension to all cones. This is in line with the corresponding question in~\cite[Section~3.3, p.~126]{Sch22:Convex-Cones} as well as~\cite[Problem~49]{Gruber1979}.
On all cones, monotonicity along a subspace flag forces
$a_0\le\cdots\le a_d$ once the intrinsic-volume representation is known,
since $v_k(L_j)=\delta_{kj}$ for linear subspaces $L_j$ with $\dim L_j=j$.
\end{remark}


Induction on dimension reduces the proof to showing that, under an
integrability condition implied by monotonicity, every simple,
$\SOn d$-invariant valuation $\eta$ on nonzero pointed polyhedral
cones is a multiple of spherical volume:
\begin{equation}\label{eq:simple-multiple}
  \eta(P)=c \sigma(P\cap \sphere{d-1}), \quad P\in \PP^*_p(\R^d),
\end{equation}
see Theorem~\ref{thm:simple-integrable}.
The continuous case corresponds to
Knoerr's Theorem~B~\cite{Knoerr:Spherical-Hadwiger} and
Schneider's Problem~74 in~\cite{Gruber1979}.

To establish~\eqref{eq:simple-multiple}, we adapt the signed simplex
and averaging ideas of Wang and Wu~\cite{WangWu:Spherical-Hadwiger}.
Finite dissections first show that a simple $\SOn d$-invariant valuation
is also invariant under reflections (Proposition~\ref{prop:polyhedral-reflections}).
We then average a signed
subdivision identity over a moving apex: reflection pairs cancel the
facet contributions, leaving only a spherical volume term. 
The crucial observation is that monotonicity of
the original valuation is enough to ensure the integrability condition needed for averaging. 
Finally, monotonicity is used in combination with 
inner and outer polyhedral approximations to extend the representation
to general pointed cones. We note that our proof also simplifies the averaging argument
in the continuous case treated by Wang and Wu.

\subsection{Historical notes and related work}\label{sec:historical-notes}

The Euclidean intrinsic volumes arise as the normalised coefficients of Steiner's
tube formula and are related to Minkowski's theory of mixed volumes;
see~\cite{Had57:Vorlesungen,Sch14:Convex-Bodies}.
Hadwiger~\cite{Had57:Vorlesungen} characterised
continuous rigid-motion-invariant valuations as their linear combinations.
The conic, or spherical, theory has roots in the
angle-sum relations of Sommerville~\cite{Som27:Angle-Sums} and their
development by McMullen~\cite{McM75:Nonlinear-Angle-Sum}.
Gr\"unbaum~\cite{Gru68:Grassmann-Angles} introduced Grassmann angles of
polytopes, while Glasauer~\cite{Glasauer95:Thesis} developed spherical
support measures and local Steiner and kinematic formulae for general
spherically convex bodies. Remarkably, Hadwiger's characterisation resisted
a generalisation to continuous, invariant valuations on the sphere until
Knoerr's recent breakthrough result~\cite{Knoerr:Spherical-Hadwiger}.

Conic intrinsic volumes also appear in several areas of applied mathematics.
In statistics, they are the mixture weights of the $\bar\chi^2$
distributions arising in likelihood-ratio tests under cone constraints;
see Takemura and Kuriki~\cite{TK97:Chi-Bar-Squared}.
Their mixed-volume representation also yields log-concavity of the conic
intrinsic volumes via the Alexandrov--Fenchel inequality, as observed by
Fu and Wang~\cite{fu2026logconcavityconicintrinsicvolumes}.
Their first moment, the statistical dimension, also enters risk bounds
for shape-constrained estimation, for example in~\cite{HWCS19:Isotonic-Regression} on isotonic regression.
In optimisation, they enter the probabilistic analysis of the conic
feasibility problem~\cite{Amelunxen:Thesis,AB15:Grassmann-Condition}.
In compressed sensing and signal recovery, intrinsic-volume concentration
and conic kinematic formulae yield sharp recovery thresholds and phase
transitions, as developed in~\cite{ALMT14:Living-Edge}.
Accounts of the conic theory and its relation
to spherical integral geometry can be found
in~\cite{AmelunxenLotz:Intrinsic-Volumes-Cones,Sch22:Convex-Cones}.

Monotonicity is a classical regularity assumption for translation-invariant
valuations on Euclidean convex bodies.  McMullen's
work~\cite{McM77:Valuations-Polytopes,McM90:Monotone-Valuations} shows
that every monotone translation-invariant real valuation on compact
convex sets is continuous in the Hausdorff metric.  Together with
Hadwiger's theorem, this gives the monotone classification under
translation and rotation invariance: on nonempty compact convex sets
$K\subseteq\R^d$, such valuations are precisely
\[
  \phi(K)=c_0+\sum_{j=1}^d c_jV_j(K),
  \qquad c_0\in\R,\quad c_1,\ldots,c_d\ge0,
\]
where $V_j$ denotes the $j$th Euclidean intrinsic volume.  The constant
term is unrestricted because monotonicity is imposed between nonempty
sets.

A related question of McMullen, recorded as
Problem~51 in~\cite{Gruber1979}, asks whether the homogeneous components
of a monotone translation-invariant valuation are themselves monotone.
Bernig and Fu~\cite[Theorem~2.12]{BF11:Hermitian-Integral-Geometry}
proved that this is indeed the case.  More recently, Ib\'a\~nez-Marcos,
Tradacete and Villanueva~\cite{ITV25:Noncontinuous-Valuations} used
automatic continuity on parallelotopes to strengthen Hadwiger's
characterisation of volume, replacing continuity by a boundedness
assumption.

The continuous and monotone spherical characterisation problems were
posed by McMullen in Problem~49 of~\cite{Gruber1979}; see
also~\cite[Section~3.3]{Sch22:Convex-Cones}.  The comment accompanying
that problem already mentions a reduction to Schneider's Problem~74 in the same volume,
which asks for a characterisation of spherical volume among continuous,
simple, rotation-invariant valuations on spherical polytopes.
Knoerr~\cite[Theorems~A and~B]{Knoerr:Spherical-Hadwiger} first settled the
continuous classification and the simple case on spherical polytopes contained
in open hemispheres, and the classification for spherical convex bodies follows
by approximation. His proof uses affine smooth valuations and the gnomonic
projection to reduce the problem to measurable, translation-invariant
valuations on Euclidean polytopes, where he applies the vanishing result
from~\cite[Proposition~5.4]{Knoerr:Polytopal-Hadwiger}.
Wang and Wu~\cite{WangWu:Spherical-Hadwiger} subsequently gave a different
proof using oriented spherical simplices and a signed coning transform.

\section{Preliminaries}\label{sec:preliminaries}
Throughout, $d\ge2$ is an integer unless otherwise stated, and
$\sphere{d-1}$ denotes the unit sphere in $\R^d$.
We write $\sigma$ for the completion of the normalised spherical measure on $\sphere{d-1}$.
We follow the notation and conventions of~\cite{Sch22:Convex-Cones}, to which we refer for more details
and background. In particular, we write $\inter$, $\cl$ and $\bd$ for interior, closure
and boundary, and $\pos$, $\lin$ and $\operatorname{conv}$ for positive,
linear and convex hulls, respectively.  Relative interior is denoted by
$\relint$, and $\inter_W$ denotes interior in a subspace $W$.

\paragraph{Convex cones.} A \emph{convex cone} is a nonempty convex set $C\subseteq\R^d$ that satisfies $\lambda C=C$ for any $\lambda>0$. We denote the set of closed convex cones in~$\R^d$ by~$\CC(\R^d)$, and we will refer to elements of $\CC(\R^d)$ simply as cones. The zero cone is $\{\zero\}$, and we write $\CC^*(\R^d)=\{C\in\CC(\R^d)\colon C\ne\{\zero\}\}$ for the set of nonzero cones.
A cone is pointed if $C\cap (-C)=\{\zero\}$, and we denote by $\CC_p^*(\R^d)$ the set of all nonzero pointed cones in $\R^d$. Denote by $\PP(\R^d)$ the set of all polyhedral cones in $\R^d$ and by $\PP_p^*(\R^d)$ the set of all nonzero pointed polyhedral cones.
We use the same notation with $\R^d$ replaced by a Euclidean subspace $W$. The dimension of a cone $C$ is denoted by $\dim C$ and is the dimension of its linear span.

\paragraph{Spherical convex sets.} We use the conic convention for
spherical convexity: $\mathcal K_s$ consists of the links
$\link C:=C\cap\sphere{d-1}$ of cones $C\in\CC(\R^d)$.
The link map is a bijection, with the zero cone corresponding to the
empty set.  Nonzero pointed cones correspond to the proper spherical
convex sets, namely the nonempty sets contained in an open hemisphere.
Spherical polytopes are links of nonzero polyhedral cones, and
$\dim(\link C)=\dim C-1$ for $C\ne\{\zero\}$.
A great hypersphere is $H\cap\sphere{d-1}$ for a linear hyperplane $H$.
For two nonempty closed subsets $X,Y\subseteq \sphere{d-1}$, the {\em Hausdorff distance} $d_H(X,Y)$ is defined as
\begin{equation*}
d_H(X,Y) := \max\left\{ \sup_{\vct{x}\in X} \inf_{\vct{y}\in Y} d_s(\vct{x},\vct{y}),\;\; \sup_{\vct{y}\in Y} \inf_{\vct{x}\in X} d_s(\vct{x},\vct{y}) \right\},
\end{equation*}
where $d_s$ denotes geodesic distance.  Set
$d_H(\varnothing,\varnothing)=0$ and $d_H(X,Y)=\infty$ when exactly one
of $X,Y$ is empty.  The resulting spherical (or angular) Hausdorff topology on
$\mathcal K_s$, and hence on $\CC(\R^d)$, has the empty set and full
sphere as isolated points.  Polyhedral cones are dense in this topology;
see \cite[Section~1.8]{Sch14:Convex-Bodies} for Euclidean polytope
approximation and \cite[Chapter~3]{Sch22:Convex-Cones} for its spherical
counterpart.

\paragraph{Valuations.} Valuations are set functionals that satisfy an inclusion--exclusion property.

\begin{definition}
A functional $\mu\colon\CC(\R^d)\to\R$ is called a (real) \emph{valuation} if
\begin{equation*}
    \mu(C_1\cup C_2)=\mu(C_1)+\mu(C_2)-\mu(C_1\cap C_2)
\end{equation*}
whenever $C_1,C_2\in\CC(\R^d)$ with $C_1\cup C_2\in\CC(\R^d)$.
\end{definition}

A valuation on $\CC(\R^d)$ is \emph{continuous} if it is continuous with respect to the spherical Hausdorff topology, \emph{monotone} if $C\subseteq D$ implies $\mu(C)\le\mu(D)$,
and it is \emph{simple} if it vanishes on every cone of dimension at most $d-1$.

For valuations on $\CC_p^*(\R^d)$, the same definitions are used with
all cones in that domain.  We may set $\mu(\{\zero\})=0$ for valuation
identities, but impose monotonicity only between nonzero cones.
Indeed, two nonzero pointed cones with convex pointed union have
nonzero intersection, so this convention does not restrict their values.

\paragraph{Intrinsic volumes and Grassmann angles.} Write $\On W$ and
$\SO(W)$ for the orthogonal and special orthogonal groups of a Euclidean
space $W$, with the usual abbreviations $\On d$ and $\SOn d$ for $W=\R^d$.
For $G\in \{\On{d}, \SOn{d}\}$, a valuation
is {\em invariant} under the action of $G$ if $\mu(gC)=\mu(C)$ for all $g\in G$ and all $C$ in its domain. Among invariant valuations, we are particularly interested in the \emph{intrinsic volumes}, $v_i\colon\CC(\R^d)\to\R$, $i\in \{0,\ldots,d\}$. There are different ways of defining the intrinsic volumes; a concise way is via Steiner's formula for a tubular neighbourhood around a spherically convex set~\cite[Chapter 4]{Sch22:Convex-Cones}. For a polyhedral cone there is a simple characterisation in terms of projections of random points. Let $\Pi_C(x)$ denote the orthogonal projection of a point $x$ onto a cone $C$. Then for $C\in \PP(\R^d)$ and $\theta$ uniformly distributed on $\sphere{d-1}$,
\begin{equation*}
  v_i(C) = \Prob\{\Pi_C(\theta)\in\relint F
                 \text{ for some $i$-dimensional face $F$ of $C$}\}.
\end{equation*}
The intrinsic volumes $v_i\colon\CC(\R^d)\to\R$, $i\in \{0,\ldots,d\}$, are linearly independent, continuous and orthogonally invariant valuations. In particular, one way of defining the intrinsic volumes on general cones is via
continuity and polyhedral approximation.
Some properties of the intrinsic volumes are:
\begin{enumerate}
\item $v_d(C)=\sigma(C\cap \sphere{d-1})$, where $\sigma$ denotes normalised $(d-1)$-dimensional spherical volume;
\item $v_k(L_j)=\delta_{kj}$ for a subspace $L_j\subseteq \R^d$ with $\dim L_j=j$, where $\delta_{kj}=1$ if $k=j$ and $\delta_{kj}=0$ otherwise;
\item the intrinsic volumes form a probability distribution: $v_i(C)\geq 0$ for all $i\in \{0,\dots,d\}$ and $\sum_{i=0}^d v_i(C)=1$;
\item if $\iota \colon\R^d\to\R^m$, $m\geq d$, is an isometric linear embedding, then $v_i(\iota(C))=v_i(C)$ (hence the name `intrinsic').
\end{enumerate}
We set $v_i(C)=0$ for indices outside $0,\ldots,d$ and note that $v_i(C)=0$
whenever $i>\dim C$.
For more definitions and details, see~\cite[Sections~2.3 and 4.2]{Sch22:Convex-Cones}.

Let $G(d,k)$ denote the Grassmannian of $k$-dimensional linear subspaces
of $\R^d$, equipped with its rotation-invariant Haar probability measure.
Let $L_{d-j}$ be a Haar-distributed $(d-j)$-dimensional linear subspace of
$\R^d$.  For $C\in\CC(\R^d)$ that is not a linear subspace and
$0\le j\le d-1$, define
\begin{equation}\label{eq:grassmann-angle}
  U_j(C):=\frac12\Prob\{C\cap L_{d-j}\ne\{\zero\}\}.
\end{equation}
In particular, $U_0(C)=1/2$ for every cone that is not a linear subspace.
For the following result, see for example \cite[Theorem~4.3.5]{Sch22:Convex-Cones} and, for
polyhedral cones, \cite[Corollary~5.2]{AmelunxenLotz:Intrinsic-Volumes-Cones}.

\begin{proposition}[Crofton's formula]\label{prop:crofton}
For every closed convex cone $C$ that is not a linear subspace and
$0\le j\le d-1$,
\begin{equation}\label{eq:crofton}
  U_j(C)=v_{j+1}(C)+v_{j+3}(C)+\cdots.
\end{equation}
\end{proposition}

We extend $U_j$ to all cones by the right-hand side of
\eqref{eq:crofton}.  For a $k$-dimensional subspace,
\begin{align}\label{eq:grassmann-subspace}
  \begin{split}
  U_j(L_k)&=1  \ \text{if $k>j$ and $k-j$ is odd},\\
  U_j(L_k)&=0  \ \text{otherwise}.
\end{split}
\end{align}

For general cones, the following properties hold:

\begin{proposition}\label{prop:grassmann-facts}
For $0\le j\le d-1$, the functional $U_j$ is a continuous,
$\On d$-invariant valuation on $\CC(\R^d)$.  If $C\subseteq D$ and
neither cone is a linear subspace, then $U_j(C)\le U_j(D)$.
For every cone $C$ that is not a linear subspace,
\begin{equation}\label{eq:dimension-positive}
  U_j(C)=0 \ \text{if }\dim C\le j,\quad U_j(C)>0 \ \text{if }\dim C\ge j+1.
\end{equation}
Moreover, the Grassmann angles are intrinsic: if $W$ is an
$m$-dimensional subspace and $C\in\CC(W)$, then
$U_j^W(C)=U_j(C)$ for $0\le j\le m-1$, where $U_j^W$ denotes the
Grassmann angle computed in $W$.
\end{proposition}

We will need linear independence of the $U_j$ on pointed cones. For orthonormal vectors $e_1,\ldots,e_d$, put
$O_k=\pos(e_1,\ldots,e_k)$.  Then
\begin{equation}\label{eq:orthant-triangular}
  U_j(O_k)=0\quad(j\ge k),\qquad
  U_{k-1}(O_k)=v_k(O_k)=2^{-k}.
\end{equation}
The last equality is the solid angle of an orthant in its span.
Thus evaluation on $O_1,\ldots,O_d$ gives an invertible triangular
matrix, proving that $U_0,\ldots,U_{d-1}$ are linearly independent on
the pointed domain.

We shall also use the following standard two-sided approximation.  It is
obtained from inscribed and circumscribed Euclidean polytope approximation
\cite[Section~1.8]{Sch14:Convex-Bodies} by applying a gnomonic projection
\cite[Section~3.2]{Sch22:Convex-Cones} in an open
hemisphere.

\begin{proposition}\label{prop:polyhedral-sandwich}
For every $C\in\CC_p^*(\R^d)$ there are sequences
$P_r,Q_r\in\PP_p^*(\R^d)$ such that
\begin{equation*}
  P_r\subseteq C\subseteq Q_r,
  \qquad P_r\rightarrow C,
  \qquad Q_r\rightarrow C
\end{equation*}
in the spherical Hausdorff topology.
\end{proposition}

\section{Simple polyhedral valuations}\label{sec:simple-polyhedral}

We first establish the finite identities and reflection symmetry needed
for averaging, and then state the precise integrability hypothesis.

\subsection{Finite cutting and extension}

For a simple valuation $\eta$, cutting by a linear hyperplane $H$ with
closed halfspaces $H^+,H^-$ gives
\begin{equation}\label{eq:simple-cut}
  \eta(P)=\eta(P\cap H^+)+\eta(P\cap H^-).
\end{equation}
This holds on the pointed domain as well as on all polyhedral cones,
with value zero assigned to the zero cone.

\begin{lemma}\label{lem:simple-extension}
Every simple valuation $\eta$ on $\PP_p^*(\R^d)$ has a unique simple
extension $\bar\eta$ to $\PP(\R^d)$.  If $\mathcal O$ is the family of
closed full-dimensional orthants of any orthonormal basis, then
\begin{equation}\label{eq:orthant-extension}
  \bar\eta(P)=\sum_{O\in\mathcal O}\eta(P\cap O).
\end{equation}
The extension preserves invariance under any subgroup of $\On d$.
\end{lemma}

\begin{proof}
Every $P\cap O$ is pointed or zero.  For pointed $P$, successive
coordinate cuts in \eqref{eq:simple-cut} show that the right-hand side
of \eqref{eq:orthant-extension} equals $\eta(P)$.  Refining two orthant
families by their common hyperplane arrangement shows that the definition
is independent of the chosen orthonormal basis.
For the valuation property, let $P_1$, $P_2$ be two polyhedral cones such that $P_1\cup P_2\in \PP(\R^d)$.
Intersecting this pair
with each $O$ gives pointed cones or the zero cone, and we apply the
valuation identity on the pointed domain, with value zero assigned to the
zero cone. Moreover, any simple extension must satisfy
\eqref{eq:orthant-extension}, proving uniqueness.  Changing the basis
by an orthogonal transformation shows that invariance is preserved. Simplicity is immediate. 
\end{proof}

We also need to pass from indicator identities to valuation identities.
The following is a simple-valuation consequence of Groemer's extension
theorem; see~\cite[Theorems~1.6.2 and~1.6.5]{Sch22:Convex-Cones}.
We include a direct proof by hyperplane cutting.

\begin{lemma}\label{lem:indicator-cutting}
Let $P_1,\ldots,P_N$ be polyhedral cones and $t_1,\ldots,t_N\in\R$.
If
\[
  \sum_{i=1}^N t_i\mathbf1_{P_i}=0
\]
outside finitely many linear hyperplanes, then every simple valuation
$\eta$ on $\PP(\R^d)$ satisfies $\sum_{i=1}^N t_i\eta(P_i)=0$.
\end{lemma}

\begin{proof}
Take a central hyperplane arrangement $H_1,\ldots,H_m$ containing the exceptional
hyperplanes, all facet hyperplanes of the full-dimensional $P_i$,
and, for each $i$ with $\dim P_i<d$, a linear hyperplane containing $P_i$.
For $\varepsilon\in\{+,-\}^m$, put
$D_\varepsilon=\bigcap_{k=1}^m H_k^{\varepsilon_k}$.
Successive application of the binary cutting identity
\eqref{eq:simple-cut}, by induction on the number of hyperplanes, gives
\[
  \eta(P_i)=\sum_{\varepsilon\in\{+,-\}^m}
    \eta(P_i\cap D_\varepsilon).
\]
At every cut, the intersection with the cutting hyperplane has value
zero by simplicity.

Let $Q_1,\ldots,Q_M$ be the closures of the full-dimensional chambers of the arrangement (that is, the 
full-dimensional $D_\varepsilon$).
For $x_j\in\inter Q_j$ set $a_{ij}=\mathbf1_{P_i}(x_j)$.
By the choice of the arrangement, each $\mathbf1_{P_i}$ is constant on
$\inter Q_j$.  If $a_{ij}=1$, then $Q_j\subseteq P_i$, and 
if $a_{ij}=0$, the intersection $P_i\cap Q_j$ is lower-dimensional.
It follows that
\[
  \eta(P_i)=\sum_{j=1}^M\eta(P_i\cap Q_j)
           =\sum_{j=1}^M a_{ij}\eta(Q_j).
\]
Each $x_j$ lies outside the exceptional hyperplanes, so the assumed
indicator identity gives $\sum_{i=1}^N t_i a_{ij}=0$ for every $j$.
Reordering the finite sums now yields
\[
  \sum_{i=1}^N t_i\eta(P_i)
  =\sum_{j=1}^M\left(\sum_{i=1}^N t_i a_{ij}\right)\eta(Q_j)
  =0.
\]
This completes the proof.
\end{proof}

In particular, a finite polyhedral subdivision with disjoint interiors
is additive for a simple valuation, including on the pointed domain
by Lemma~\ref{lem:simple-extension}.  Every nonzero pointed polyhedral
cone admits a simplicial subdivision (cf.~\cite[Section~2.6, p.~48]{Fulton93:Toric-Varieties}).
Consequently, a simple valuation on pointed polyhedral cones is determined
by its values on full-dimensional simplicial cones, so equality of two
such valuations need only be checked on this class.

\subsection{Reflection pairing}

The following incentre dissection is the conic version of a classical
construction in scissors congruence; see~\cite{Sah79:Hilberts-Third}.
We include its proof for completeness. The use of such a theorem to show that
$\SOn d$-invariant valuations are $\On d$-invariant was pointed out by Klain and Rota~\cite[Section~11.2]{KR97:Introduction-Geometric}.

\begin{lemma}\label{lem:reflection-pairing}
Every full-dimensional pointed simplicial cone $C$ has a subdivision
into $d(d-1)$ simplicial cones $B_{ij}$, $i\ne j$, with disjoint
interiors, such that a hyperplane reflection interchanges $B_{ij}$
and $B_{ji}$.
\end{lemma}

\begin{proof}
Write $C=\pos(p_1,\ldots,p_d)$ for a basis $p_1,\ldots,p_d$, let
$y_1,\ldots,y_d$ be its dual basis, and set
\[
  F_i=\pos(p_k:k\ne i),\qquad
  \nu_i=\frac{y_i}{\|y_i\|},\qquad
  z=\sum_{k=1}^d\|y_k\|p_k,\qquad z_i=z-\nu_i.
\]
Here $\nu_i$ is the inward unit normal to $F_i$,
$z\in\inter C$, and $\langle\nu_i,z\rangle=1$.
Furthermore,
\[
  \langle y_k,z_i\rangle
  =\|y_k\|\bigl(1-\langle\nu_k,\nu_i\rangle\bigr).
\]
This is zero for $k=i$ and strictly positive otherwise, since distinct
dual basis vectors have distinct unit directions.  Hence
$z_i\in\relint F_i$.
Star subdivision at $z$, followed by star subdivision of each $F_i$
at $z_i$, gives a subdivision into the cones
\[
  B_{ij}=\pos\bigl(z,z_i,p_k:k\notin\{i,j\}\bigr),\qquad i\ne j.
\]
For $d=2$ the second subdivision is the identity on each ray.
Reflection $s_{ij}$ in $(\nu_i-\nu_j)^\perp$ interchanges $\nu_i$
and $\nu_j$, fixes $z$, and fixes every $p_k$ with $k\notin\{i,j\}$.
Thus $s_{ij}z_i=z_j$ and $s_{ij}B_{ij}=B_{ji}$.
Figure~\ref{fig:incenter-dissection} illustrates the construction for $d=3$.
\end{proof}

\begin{figure}[htbp]
  \centering
  \includegraphics[width=0.60\linewidth]{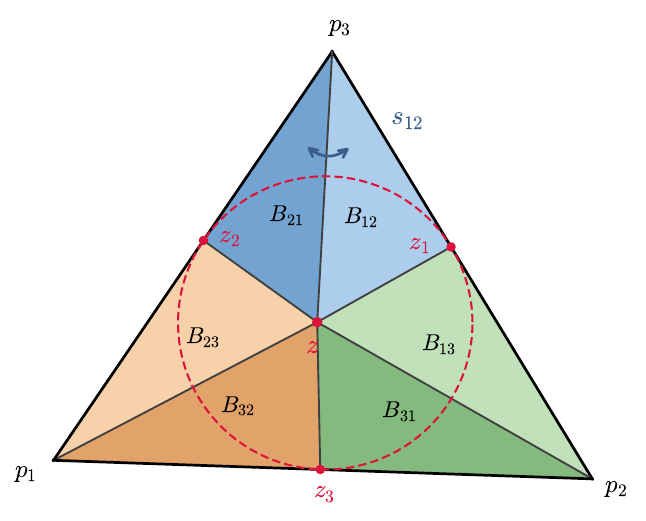}
  \caption{An affine section of the dissection in
    Lemma~\ref{lem:reflection-pairing} for $d=3$.}
  \label{fig:incenter-dissection}
\end{figure}

\begin{proposition}\label{prop:polyhedral-reflections}
Every simple $\SOn d$-invariant valuation on $\PP_p^*(\R^d)$ is
$\On d$-invariant.
\end{proposition}

\begin{proof}
Let $g$ reverse orientation and let $C$ be full-dimensional and
simplicial.  Lemma~\ref{lem:reflection-pairing} gives
\[
  \eta(gC)=\sum_{i\ne j}\eta(gB_{ij})
          =\sum_{i\ne j}\eta((gs_{ij})B_{ji})
          =\sum_{i\ne j}\eta(B_{ji})=\eta(C).
\]
The middle equality uses $\det(gs_{ij})=1$.
Subdivision proves the result for all pointed polyhedral cones.
\end{proof}

\subsection{Signed stellar subdivision}

A key ingredient in our proof is a signed simplex identity underlying Wang and Wu's cocycle argument
\cite[Lemma~6.2]{WangWu:Spherical-Hadwiger}. We give a short proof based on
the interval of nonnegative representations of a vector.
Here and throughout we use the usual convention that an entry marked with a hat is omitted.

\begin{lemma}\label{lem:circuit}
Let $q_0,\ldots,q_d\in\R^d$ be such that every $d$ of them are linearly
independent. Set
\[
  \alpha_i=(-1)^i\det(q_0,\ldots,\widehat q_i,\ldots,q_d),
  \qquad C_i=\pos(q_0,\ldots,\widehat q_i,\ldots,q_d).
\]
Then outside the spans of
all $(d-1)$-element subsets of the $q_i$,
\begin{equation}\label{eq:circuit}
  \sum_{i=0}^d\sgn(\alpha_i)\mathbf1_{C_i}=
  \begin{cases}
    1,&\alpha_i>0\text{ for all }i,\\
    -1,&\alpha_i<0\text{ for all }i,\\
    0,&\text{otherwise}.
  \end{cases}
\end{equation}
\end{lemma}

\begin{proof}
Let $Q=[q_0\ \cdots\ q_d]$ and $\alpha=(\alpha_0,\ldots,\alpha_d)^T$.
Laplace expansion gives $Q\alpha=0$.  Since $Q$ has rank $d$ and
every $\alpha_i$ is nonzero, $\ker Q=\R\alpha$.
Fix $y$ outside the exceptional spans and choose $a\in\R^{d+1}$ with
$Qa=y$.  All nonnegative representations of $y$ are therefore
parametrised by the closed interval
\[
  I_y=\{t\in\R:a_i+t\alpha_i\ge0\text{ for every }i\},
\]
which may be empty or unbounded.  At any $t\in I_y$, at most one
coordinate of $a+t\alpha$ vanishes, since two vanishing coordinates
would put $y$ in the span of $d-1$ generators.  In particular, $I_y$
cannot be a singleton: this would require both a lower and an upper
bound to be attained, hence two vanishing coordinates.

By definition,
\[
  y\in C_i\quad\Longleftrightarrow\quad
  a_i+t\alpha_i=0\text{ for some }t\in I_y.
\]
Such a $t$ is a finite endpoint of $I_y$, and every finite endpoint
has exactly one vanishing coordinate.  A lower endpoint corresponds
to an index $i$ with $\alpha_i>0$, and an upper endpoint to an index
$j$ with $\alpha_j<0$.

If all $\alpha_i>0$, then $I_y$ is a half-line with one lower endpoint,
so the sum in~\eqref{eq:circuit} is $1$.  If all $\alpha_i<0$, it is
a half-line with one upper endpoint, giving $-1$.  If the signs are
mixed, $I_y$ is either empty or a nondegenerate bounded interval;
in the latter case its two endpoints contribute opposite signs.
Thus the sum is zero, proving~\eqref{eq:circuit}.
\end{proof}

\begin{remark}\label{rem:circuit-coverage}
The proof does not require the $d+1$ points to be affinely independent.
If they are, put $T=\operatorname{conv}(q_0,\ldots,q_d)$.
All $\alpha_i$ have the same sign if and only if $0\in\inter T$.
Consider rays from the origin whose directions avoid the exceptional
spans in Lemma~\ref{lem:circuit}.  If $0\in\inter T$, each such ray
exits $T$ through the relative interior of one facet.  If $0\notin T$,
each such ray either misses $T$ or enters and exits through the relative
interiors of two distinct facets, whose contributions have opposite signs.

In the language of oriented matroids, the sign
vector $(\sgn(\alpha_i))_{i=0}^d$ is a signed circuit of the vector
configuration represented by $Q$, and replacing selected generators
by their negatives is a reorientation; see~\cite[Section~1.2]{BLSWZ99:Oriented-Matroids}.
\end{remark}

\begin{proposition}\label{prop:stellar}
Let $\eta$ be simple on $\PP_p^*(\R^d)$, with extension $\bar\eta$ from
Lemma~\ref{lem:simple-extension}.  For a basis $p_1,\ldots,p_d$, put
\[
  C=\pos(p_1,\ldots,p_d),\qquad F_j=\pos(p_i:i\ne j),\qquad
  x=\sum_{j=1}^d a_j(x)p_j.
\]
If all $a_j(x)\ne0$, then
\begin{equation}\label{eq:stellar}
  \eta(C)=\mathbf1_{-C}(x)\bar\eta(\R^d)
       +\sum_{j=1}^d\sgn(a_j(x))\eta(F_j+\R_{\ge0}x).
\end{equation}
\end{proposition}

\begin{proof}
Apply Lemma~\ref{lem:circuit} to $(x,p_1,\ldots,p_d)$.
Writing $D=\det(p_1,\ldots,p_d)$, its cofactor vector is
$\alpha=D\cdot (1,-a_1(x),\ldots,-a_d(x))^T$.
The cofactor signs are all equal precisely when all $a_j(x)<0$,
or equivalently $x\in -C$.  Dividing the identity in
Lemma~\ref{lem:circuit} by $\sgn(D)$ and rearranging gives
\begin{equation*}
 \mathbf1_{C} = \mathbf1_{-C}(x)\mathbf1_{\R^d}+\sum_{j=1}^d\sgn(a_j(x))\mathbf1_{(F_j+\R_{\ge0}x)}.
\end{equation*}
This holds outside the spans of all $(d-1)$-element subsets of
$(x,p_1,\ldots,p_d)$, hence outside finitely many linear hyperplanes.
Lemma~\ref{lem:indicator-cutting} applied to $\bar\eta$
gives \eqref{eq:stellar}.
\end{proof}

\subsection{Averaging the apex}

The following theorem gives an affirmative answer to Schneider's
Problem~74 in Gruber and Schneider~\cite[Problem~74, p.~272]{Gruber1979},
with continuity replaced by integrability of the fixed-facet apex
functions.  In the continuous case, these functions are continuous
away from the exceptional great hypersphere and bounded.  Indeed,
cutting the simplicial cones into coordinate orthants gives pieces
contained in compact families of pointed polyhedral cones with
uniformly bounded numbers of extreme rays. Continuity and simplicity
then give a uniform bound.

\begin{theorem}\label{thm:simple-integrable}
Let $\eta$ be a simple $\SOn d$-invariant valuation on
$\PP_p^*(\R^d)$.  Suppose that for every pointed simplicial cone $F$
of dimension $d-1$, the function
\[
  x\longmapsto\eta(F+\R_{\ge0}x)
\]
is integrable on $\sphere{d-1}\setminus\lin F$.
Then
\[
  \eta(P)=\bar\eta(\R^d)v_d(P)
  \qquad\text{for all }P\in\PP_p^*(\R^d).
\]
\end{theorem}

\begin{proof}
By Proposition~\ref{prop:polyhedral-reflections}, $\eta$ is orthogonally
invariant.  Fix a full-dimensional simplicial cone $C$ and use the
notation of Proposition~\ref{prop:stellar}.
Write $x=\sum_{j=1}^d a_j(x)p_j$, where the $a_j$ are the coordinate
functionals of the chosen basis $p_1,\ldots,p_d$.
Let $r_j$ be reflection in $H_j=\lin F_j$.
It fixes $F_j$ pointwise, and the linear functional $a_j$ vanishes
on $H_j$, so
\[
  a_j(r_jx)=-a_j(x),\qquad
  F_j+\R_{\ge0}r_jx=r_j(F_j+\R_{\ge0}x).
\]
Thus the $j$th signed summand in \eqref{eq:stellar} is odd under
$r_j$ and, by the integrability hypothesis, has integral zero.
The exceptional apices in \eqref{eq:stellar} lie in the finitely many
great hyperspheres $H_j\cap\sphere{d-1}$.  Integrating gives
\[
  \eta(C)=\bar\eta(\R^d)\sigma(-C\cap\sphere{d-1})
         =\bar\eta(\R^d)v_d(C).
\]
Simplicity and subdivision complete the proof.
\end{proof}

The theorem uses only integrability of each fixed-facet apex function,
not joint measurability in all generators.  It determines polyhedral
values only.  The passage to general cones will use monotonicity of
the original valuation, not continuity or monotonicity of its simple
remainder.

\section{Monotone valuations on pointed cones}\label{sec:monotone-cone}

We now establish the boundedness and measurability needed for
Theorem~\ref{thm:simple-integrable}, and then prove the representation
and automatic continuity together.

\begin{lemma}\label{lem:normalization}
Let $\mu$ be a monotone $\SOn d$-invariant valuation on
$\CC_p^*(\R^d)$, and let $\alpha$ be its common value on rays.
Then $\mu_0:=\mu-2\alpha U_0$ is monotone, invariant, nonnegative,
bounded, and vanishes on rays.
\end{lemma}

\begin{proof}
Since $2\alpha U_0=\alpha$ on nonzero pointed cones, subtracting it
preserves monotonicity.  Every nonzero cone contains a ray, so
$\mu_0\ge0$.  Set $\mu_0(\{\zero\})=0$.
For a linear hyperplane $H$ with closed halfspaces $H^+,H^-$,
\[
  \mu_0(C)
  =\mu_0(C\cap H^+)+\mu_0(C\cap H^-)-\mu_0(C\cap H)
  \le\mu_0(C\cap H^+)+\mu_0(C\cap H^-).
\]
Successive cuts by the coordinate hyperplanes give, for their full
orthant family $\mathcal O$,
\[
  0\le\mu_0(C)
  \le\sum_{O\in\mathcal O}\mu_0(C\cap O)
  \le\sum_{O\in\mathcal O}\mu_0(O).
\]
The right-hand side is a finite constant independent of $C$.
\end{proof}

Let $H\subset\R^d$ be a linear hyperplane.  For
$C\in\CC_p^*(H)$ and $x\in\sphere{d-1}\setminus H$, define
\begin{equation}\label{eq:coning-operation}
  x*C:=C+\R_{\ge0}x.
\end{equation}
The linear isomorphism $H\times\R\to\R^d$,
$(q,t)\mapsto q+tx$, identifies $x*C$ with $C\times[0,\infty)$,
so $x*C$ is closed and pointed.
For fixed $C$, it depends continuously on $x\notin H$:
the corresponding linear isomorphisms converge, and their induced
maps on the sphere converge uniformly.

\begin{lemma}\label{lem:measure}
Let $m\ge1$, let $C\in\CC_p^*(\R^m)$ be full-dimensional, and let
$A\subseteq\R^m$ satisfy $A+\inter C\subseteq A$.
Then $A$ is Lebesgue measurable.
\end{lemma}

\begin{proof}
The set $A+\inter C$ is open and contained in $A$, hence in $\inter A$.
Openness of $\inter C$ also implies that
\begin{equation}\label{eq:inter-closure}
  \cl A+\inter C\subseteq\inter A.
\end{equation}
Indeed, for $a\in\cl A$ and $h\in\inter C$, choose $a'\in A$
sufficiently close to $a$ that $h+a-a'\in\inter C$.
Then $a+h=a'+(h+a-a')\in\inter A$.
Fix a unit vector $e\in\inter C$.  Each line parallel to $e$ meets
the closed boundary $\bd A$ in at most one point: if $a,a+te\in\bd A$
with $t>0$, \eqref{eq:inter-closure} would put $a+te$ in $\inter A$.
Writing $\lambda_k$ for $k$-dimensional Lebesgue measure,
Fubini--Tonelli gives
\[
  \lambda_m(\bd A)
  =\int_{e^\perp}
     \lambda_1\bigl(\{t\in\R:y+te\in\bd A\}\bigr)
       \,\mathrm{d}\lambda_{m-1}(y)=0.
\]
Here $\lambda_0$ is unit mass on the zero-dimensional space when $m=1$.
Since $A\setminus\inter A\subseteq\bd A$, completeness of Lebesgue
measure makes $A$ measurable.
\end{proof}

\begin{lemma}\label{lem:measurable-coning}
Let $H\subset\R^d$ be a linear hyperplane and let $\mu$ be a monotone
valuation on $\CC_p^*(\R^d)$.
If $C\in\CC_p^*(H)$ is full-dimensional in $H$, then
$x\mapsto\mu(x*C)$ is measurable on $\sphere{d-1}\setminus H$
for completed spherical measure.
\end{lemma}

\begin{proof}
Choose a unit normal $n_H$ to $H$ and parametrise the open hemispheres by
\[
  x_\pm(v)=\frac{v\pm n_H}{\|v\pm n_H\|},\qquad
  f_\pm(v)=\mu(x_\pm(v)*C),\qquad v\in H.
\]
For $h\in C$, the inclusion
\[
  C+\R_{\ge0}(v+h\pm n_H)\subseteq C+\R_{\ge0}(v\pm n_H)
\]
gives $f_\pm(v+h)\le f_\pm(v)$.
Thus each sublevel set of $f_\pm$ is closed under addition of elements of
$\inter_H C$.  Lemma~\ref{lem:measure}, applied in $H$, makes these
sublevel sets Lebesgue measurable.  The hemisphere charts are smooth
diffeomorphisms and transfer completed Lebesgue measurability to
completed spherical measurability.
\end{proof}

\begin{theorem}\label{thm:regularity}
Every monotone $\SOn d$-invariant valuation on $\CC_p^*(\R^d)$
has a unique representation
\begin{equation}\label{eq:pointed-representation}
  \mu=\sum_{j=0}^{d-1}a_jU_j.
\end{equation}
In particular, it is continuous and $\On d$-invariant.
\end{theorem}

\begin{proof}
Uniqueness follows from \eqref{eq:orthant-triangular}.
We prove existence by induction on $d$, normalising $\mu$ by
Lemma~\ref{lem:normalization} to be bounded, nonnegative, and zero
on rays.

For $d=2$, set $\nu=0$: the normalised valuation already vanishes
on all lower-dimen\-sional pointed cones.
For $d\ge3$, assume the representation in dimension $d-1$ and
construct $\nu$ as follows.
For a hyperplane $W$, the restriction of $\mu$ to $\CC_p^*(W)$ is
monotone and $\SO(W)$-invariant, since rotations of $W$ extend to
rotations of $\R^d$ fixing its normal line pointwise.
By induction,
\[
  \mu|_{\CC_p^*(W)}=\sum_{j=1}^{d-2}b_jU_j^W.
\]
The constant coefficient is zero because $\mu$ vanishes on rays.
Uniqueness and rotation invariance make the coefficients independent
of $W$.  Intrinsicness of the Grassmann angles shows that
\[
  \nu:=\sum_{j=1}^{d-2}b_jU_j
\]
agrees with $\mu$ on every lower-dimensional nonzero pointed cone.

In either case, $\eta:=\mu-\nu$ is simple, bounded, and
$\SOn d$-invariant.
Here boundedness of $\nu$ follows directly from its intrinsic-volume
representation.  Neither monotonicity nor nonnegativity of $\eta$
is assumed.

For every pointed simplicial facet $F$,
Lemma~\ref{lem:measurable-coning} makes $x\mapsto\mu(x*F)$ measurable.
The function $x\mapsto\nu(x*F)$ is continuous off $\lin F$.
Thus the apex functions for $\eta$ are measurable and bounded.
Theorem~\ref{thm:simple-integrable} gives a constant $c$ such that
\[
  \mu(P)=g(P),\qquad
  g:=\nu+cv_d=\nu+cU_{d-1},
  \qquad P\in\PP_p^*(\R^d).
\]

For $C\in\CC_p^*(\R^d)$, take the approximations from
Proposition~\ref{prop:polyhedral-sandwich}.  Monotonicity of the
original $\mu$ gives
\begin{equation}\label{eq:monotone-sandwich}
  g(P_r)=\mu(P_r)\le\mu(C)\le\mu(Q_r)=g(Q_r).
\end{equation}
Both bounds converge to $g(C)$ by continuity of the known function $g$.
Hence $\mu(C)=g(C)$.  Restoring the constant removed in the
normalisation proves \eqref{eq:pointed-representation} and completes
the induction.  Continuity and orthogonal invariance follow from
the corresponding properties of the $U_j$.
\end{proof}

We now determine the coefficient signs.

\begin{proposition}\label{prop:coefficient-signs}
The valuation $\mu=\sum_{j=0}^{d-1}a_jU_j$ on $\CC_p^*(\R^d)$
is monotone if and only if $a_j\ge0$ for $1\le j\le d-1$.
\end{proposition}

\begin{proof}
Sufficiency follows because $U_0$ is constant on nonzero pointed
cones and the other $U_j$ are monotone there.
Conversely, let $\widetilde\mu=\sum_{j=0}^{d-1}a_jU_j$ be the
continuous extension defined by the intrinsic-volume formulae.
Choose a complete flag $E_1\subset\cdots\subset E_d=\R^d$ and a unit
vector $e\in E_1$.  The relative halfspaces
\[
  H_m=\{x\in E_m:\langle x,e\rangle\ge0\},\qquad 1\le m\le d,
\]
have $v_{m-1}(H_m)=v_m(H_m)=1/2$ and all other intrinsic volumes zero.
Indeed, the valuation property and orthogonal invariance give
$2v_i(H_m)=v_i(E_m)+v_i(E_m\cap e^\perp)$.
The intrinsic-volume formula \eqref{eq:crofton} therefore gives
\[
  U_i(H_m)=
  \begin{cases}\tfrac12,&i<m,\\0,&i\ge m.\end{cases}
\]
For $0<\eps<1$, the pointed cones
\[
  H_{m,\eps}=\{x\in E_m:\langle x,e\rangle\ge\eps\|x\|\}
\]
are nested in $m$ and converge to $H_m$ as $\eps\rightarrow 0$.
Monotonicity on these pointed cones and continuity of
$\widetilde\mu$ imply
\[
  0\le\widetilde\mu(H_{j+1})-\widetilde\mu(H_j)
    =\frac{a_j}{2},\qquad 1\le j\le d-1.\qedhere
\]
\end{proof}

\begin{proof}[Proof of Corollary~\ref{cor:problem49}]
Combine Theorem~\ref{thm:regularity} and
Proposition~\ref{prop:coefficient-signs}.
\end{proof}

\begin{remark}\label{rem:zero-monotonicity}
If the valuation is extended by $\mu(\{\zero\})=0$ and monotonicity
is also required for $\{\zero\}\subset C$, then one must additionally
have $a_0\ge0$.  In that convention the monotone cone is generated by
$U_0,\ldots,U_{d-1}$ with nonnegative coefficients.
\end{remark}

An earlier theorem of Schneider classifies nonnegative simple invariant
valuations on proper spherical polytopes
\cite[Theorem~6.2]{Sch78:Curvature-Measures},
\cite[Theorem~3.3.2]{Sch22:Convex-Cones}.
It cannot be applied directly to the simple remainder in our induction,
since $\mu-\nu$ need not be nonnegative.
Theorem~\ref{thm:simple-integrable} replaces this sign condition by
the apex integrability implied by monotonicity.

\section{Classification on all cones}\label{sec:all-cones}

On the full cone space, the natural monotone coordinates are the
intrinsic-volume tails
\begin{equation}\label{eq:tail-functional}
  W_j:=\sum_{k=j}^{d}v_k,\qquad 0\le j\le d.
\end{equation}
They are continuous $\On d$-invariant valuations, with $W_0=1$.
For $1\le j\le d$, their standard projection representation is
\begin{equation}\label{eq:tail-projection}
  W_j(C)=\int_{G(d,j)}
     \gamma_L\bigl(\cl(C\vert L)\bigr)\,\mathrm{d}L,
\end{equation}
where $C\vert L$ denotes orthogonal projection, $\gamma_L$ is normalised
solid angle in $L$, and $\mathrm{d}L$ is Haar probability measure
\cite[equations~(2.76)--(2.77) and (4.70)]{Sch22:Convex-Cones}.
Projection, closure and solid angle preserve inclusions, so every $W_j$ is
monotone.

In contrast, the Grassmann angles need not be globally monotone,
owing to their subspace values \eqref{eq:grassmann-subspace}.
For arbitrary coefficients,
\begin{equation}\label{eq:tail-expansion}
  \sum_{k=0}^{d}a_kv_k
  =a_0W_0+\sum_{j=1}^{d}(a_j-a_{j-1})W_j.
\end{equation}

\begin{proof}[Proof of Theorem~\ref{thm:global-classification}]
Fix an orthonormal basis $e_1,\ldots,e_d$ of $\R^d$.
Set $L_0:=\{\zero\}$ and $L_k:=\lin(e_1,\ldots,e_k)$ for $1\le k\le d$.
If $\mu$ is monotone and invariant, set $a_k:=\mu(L_k)$ and
\[
  \mu_c:=\sum_{k=0}^{d}a_kv_k.
\]
The coefficients are nondecreasing, and $v_i(L_k)=\delta_{ik}$ implies
$\mu_c(L_k)=\mu(L_k)$.  Formula~\eqref{eq:tail-expansion} shows that
$\mu_c$ is monotone.  Applying Theorem~\ref{thm:regularity} separately
to the pointed restrictions of $\mu$ and $\mu_c$, we obtain
\[
  \zeta:=\mu-\mu_c=\sum_{j=0}^{d-1}c_jU_j
  \quad\text{on }\CC_p^*(\R^d).
\]
We show that these coefficients vanish by evaluating on orthants.

Let $O_k:=\pos(e_1,\ldots,e_k)$ for $1\le k\le d$, with $O_0:=L_0$.
Repeated coordinate-hyperplane cutting in $L_k$ gives
\begin{equation}\label{eq:orthant-recovery}
  0=\zeta(L_k)
   =\sum_{i=0}^{k}(-1)^{k-i}2^i\binom{k}{i}\zeta(O_i).
\end{equation}
Here $2^i\binom{k}{i}$ counts the $i$-dimensional orthants, which are
congruent under $\SOn d$. For full-dimensional orthants an odd
coordinate permutation corrects the determinant if necessary.
Since $\zeta(O_0)=0$, induction on $k$ in
\eqref{eq:orthant-recovery} yields $\zeta(O_k)=0$ for every $k$.
The triangular relations \eqref{eq:orthant-triangular} now give
$c_0=\cdots=c_{d-1}=0$.

Thus $\zeta$ vanishes on the zero cone and on every nonzero pointed
cone.  To conclude, cut any $C\in\CC(\R^d)$ successively by all
coordinate hyperplanes.  Every term in the resulting valuation identity
is $C$ intersected with a coordinate orthant or one of its faces, hence
is pointed or zero.  All terms vanish, so $\zeta(C)=0$.
This proves $\mu=\mu_c$ on the full cone space and, in particular,
automatic continuity.

Conversely, if $a_0\le\cdots\le a_d$, then
\eqref{eq:tail-expansion} is a constant plus a nonnegative combination
of monotone tails.
Continuity, the valuation property, and orthogonal invariance follow
from those of the intrinsic volumes.
\end{proof}

\section*{Acknowledgements}
The author would like to thank Dennis Amelunxen for insightful discussions on
intrinsic volumes and valuations, and Matthias Reitzner and Fabian Mu\ss nig for
pointing out the remarkable work by Jonas Knoerr.

\section*{Statement on the use of generative AI}
The author made use of OpenAI's Codex as follows: general mathematical discussions,
checking the literature, proofreading and corrections, and the details of the 
measurability argument in Lemma~\ref{lem:measure}.
Unless otherwise stated, the ideas are the author's own and the author is responsible 
for the content of this paper.

\begingroup
\small
\bibliographystyle{alpha}
\bibliography{refs}
\endgroup

\end{document}